\documentclass[12pt]{amsart}
\usepackage{amsmath,amsthm,amssymb}
\usepackage{hyperref}
\hypersetup{colorlinks,
citecolor=red,
linkcolor=blue,
urlcolor=black}

\newcommand{\al}{\alpha}

\newcommand{\fr}{\mathcal{F}}
\newcommand{\ity}{\infty}
\newcommand{\C}{\mathbb{C}}

\newcommand{\N}{\mathbb{N}}

\numberwithin{equation}{section}
\newtheorem{theorem}{Theorem}[section]
\newtheorem{lemma}[theorem]{Lemma}
\newtheorem{corollary}[theorem]{Corollary}
\newtheorem*{thma}{Theorem A}
\newtheorem*{thmb}{Theorem B}
\newtheorem*{thmc}{Theorem C}

\newtheorem*{thme}{Theorem D}
\newtheorem*{thme'}{Theorem D$'$}
\newtheorem*{zl}{Zalcman's Lemma}
\theoremstyle{remark}
\newtheorem{remark}[theorem]{Remark}
\newtheorem{example}[theorem]{Example}

\newtheorem*{obs}{Observation}
\newtheorem*{exam}{Example}
\makeatletter
\@namedef{subjclassname@2010}{%
  \textup{2010} Mathematics Subject Classification}
\makeatother

\begin{document}

\title[Normality and Montel's Theorem ]{Normality and Montel's Theorem}

\thanks{The research work  of the first author is supported by research fellowship from UGC India.}

\author[G. Datt]{Gopal Datt}
\address{Department of Mathematics, University of Delhi,
Delhi--110 007, India} \email{ggopal.datt@gmail.com \\ gdatt1@maths.du.ac.in}

\author[S. Kumar]{Sanjay Kumar}

\address{Department of Mathematics, Deen Dayal Upadhyaya College, University of Delhi,
Delhi--110 015, India }
\email{sanjpant@gmail.com}

\begin{abstract}
In this article, we prove a normality criterion for a family of meromorphic functions having zeros with some multiplicity   which involves sharing of a holomorphic function  by the  members of the family. Our result generalizes Montel's normality test in a certain sense.
\end{abstract}
 \maketitle
\section{Introduction and main results}
The notion  of normal families has played a key role in the progress of function theory.  The convergence of a family of functions always has far reaching consequences.  The concept of local convergence of a sequence of functions was introduced by Montel who later gave the notion of normal family.  He gave a result on the convergence of the sequence of  holomorphic functions which says that a sequence of uniformly bounded  holomorphic functions has  a subsequence that is locally uniformly convergent. Let us recall the definition: {\it A family of meromorphic (holomorphic) functions defined on a domain $D\subset \C$ is said  to be \emph{normal} in the domain, if every sequence in the family  has a subsequence which converges spherically  uniformly on compact subsets of  $D$ to a meromorphic (holomorphic) function or to $\ity$.} \cite{Ahl, Hay, Schiff, Yang}.\\

The most celebrated result in the theory of normal families is  Montel's Crit$\grave{\text{e}}$re Fondamental (Fundamental Normality Test), which says that: {\it A family $\mathcal F$ of meromorphic functions in a domain $D\subset \C$, which  omits three distinct complex numbers, is normal in the domain $D$.} This result supports  Bloch's heuristic  principle which says that a family $\fr$ of meromorphic functions endowed with a property $P$ is normal if  condition $P$ reduces a meromorphic function to a constant in the plane. Although this principle is not true in general, many researchers gave normality criteria for families of meromorphic functions supporting Bloch's heuristic principle.  Inspired by Bloch's principle, Schwick  discovered a connection between shared values and normality \cite{Sch 92}. Since then many researchers proved normality criteria concerning shared values \cite{GDSK,GDSK1,Sun 94,Xu 10}. Let us recall the meaning of shared values.
Let $f$ be a meromorphic function of a domain $D\subset\C$. For $p\in \C$, let
\begin{equation*}
  E_f(p)=\{z\in D: f(z)=p\}
\end{equation*}
and let
\begin{equation*}
E_f(\infty)= \text{poles of\ } f \text{\ in\ } D.
\end{equation*}
 For $p\in \C\cup \{\infty\}$, two meromorphic functions $f$ and $g$ of $D$ share the value $p$ if $E_f(p)=E_g(p).$

Improving the fundamental normality test Sun \cite{Sun 94} proved the following theorem.
\begin{thma}\cite{Sun 94}\label{Sun 94}
  Let $\fr$ be a family of functions meromorphic in a plane domain $D$. If  each pair of functions $f$ and $g$ share $0, 1, \infty$, then $\fr$ is normal in $D$.
\end{thma}
This result of Sun was further improved by Xu \cite{Xu 10} as follows.
\begin{thmb}\cite{Xu 10}\label{Xu1 10}
  Let $\fr$ be a family of functions meromorphic in a plane domain $D$. Suppose that
  \begin{enumerate}
    \item {$f$ and $g$ share $0$ in $D$, for each pair $f$ and $g$ in $\fr$,}
    \item {all zeros of $f-1$ are of multiplicity at least $3\  (\text{or}\ 2)$, for each $f\in\fr$,}
    \item {all poles of $f$ are of multiplicity at least $2 \ (\text{or}\ 3)$, for each $f\in\fr$,}
  \end{enumerate}
  then $\fr$ is normal in $D$.
\end{thmb}
In the same paper Xu also proved the following normality criterion.
\begin{thmc}\cite{Xu 10}\label{Xu2 10}
 Let $\fr$ be a family of functions meromorphic in a plane domain $D$ and let $\psi\not\equiv 0, \infty$ be a meromorphic function in $D$. Suppose that
  \begin{enumerate}
    \item {$f$ and $g$ share $0, \infty, \psi(z)$ in $D$, for each pair $f$ and $g$ in $\fr$,}
    \item {the multiplicity of $f\in\fr$ is larger than that of $\psi(z)$ at the common zeros or poles of $f$ and $\psi(z)$ in $D$, }
   \end{enumerate}
  then $\fr$ is normal in $D$.
\end{thmc}

\begin{obs} We observe that we can not assure normality in case each pair $f, g$ of $\fr$ shares $0, \infty$. We have the following example supporting  this observation.
\begin{exam}
  Let $D=\{z:|z|<1\}$ and $\fr=\{nz: n\in \N\}.$ Clearly  each pair $f,\ g$ of $\fr$ shares $0, \infty$ but $\fr$ is not normal in $D$.
\end{exam}
This example also confirms that normality will no longer be assured if each pair $f, g$ of $\fr$ shares a holomorphic function which is identically $0$ in $D$.
\end{obs}
It is natural to ask  if  we can assure normality after removing  the condition of sharing  $0$ and $\infty$ in Theorem C$?$ In this paper we  discuss this  problem and propose a normality criterion where a holomorphic function is shared by each pair of  functions of the family. Let us recall the definition of shared function. We say two functions $f$ and $g$ share a function $h$ IM in a domain $D$, if $\{z\in D:f(z)=h(z)\}=\{z\in D:g(z)=h(z)\}$.
We obtain the following result which clearly generalizes Theorem B.

\begin{theorem}\label{Main1 function thesis}
  Let $\fr$ be a family of meromorphic functions defined on a domain $D\subset\C$ and let $\psi\not\equiv0$ be a holomorphic function in $D$ such that zeros of $\psi(z)$ are of multiplicity at most $m$. Suppose that
  \begin{enumerate}
    \item {all poles of $f$ are of multiplicity at least $3(m+1)\ (\text{or}\ 2(m+1))$,}
    \item { all zeros of $f$ are of multiplicity at least $2(m+1)\ (\text{resp.}\ 3(m+1) )$, }
    \item { each pair $f$ and $g$ of $\fr$ shares $\psi$ IM in $D$,}
  \end{enumerate}
  then $\fr$ is normal in $D$.
\end{theorem}
\begin{corollary}\label{Cor}
  Let $\fr$ be a family of meromorphic functions defined on a domain $D\subset\C$.  Suppose that
  \begin{enumerate}
    \item {all poles of $f$ are of multiplicity at least $3\ (\text{or}\ 2)$,}
    \item { all zeros of $f$ are of multiplicity at least $2\ (\text{resp.}\ 3 )$, }
    \item { each pair $f$ and $g$ of $\fr$ shares $1$ IM in $D$,}
  \end{enumerate}
  then $\fr$ is normal in $D$.

\end{corollary}
\begin{corollary}
  Let $\fr$ be a family of holomorphic functions defined on a domain $D\subset\C$ and let $\psi\not\equiv 0$ be a holomorphic function in $D$ such that zeros of $\psi(z)$ are of multiplicity at most $m$. Suppose that
  \begin{enumerate}
        \item { all zeros of $f$ are of multiplicity at least $2(m+1)$, }
    \item { each pair $f$ and $g$ of $\fr$ shares $\psi$ IM in $D$,}
  \end{enumerate}
  then $\fr$ is normal in $D$.
\end{corollary}
\begin{remark}
  It is easy to see that Corollary \ref{Cor} can also  be obtained from Theorem B, by considering the family $\fr_1=\{1-f: f\in \fr\}$.
\end{remark}
The following example shows that the condition on the multiplicities of zeros in Corollary \ref{Cor} (and hence also in Theorem \ref{Main1 function thesis}) is necessary.
\begin{example}
  Let $D=\{z:|z|<1\}$ and $\fr=\{nz+1: n\in \N\}.$ Let $\psi(z)\equiv1$. Then $m=0$ and each pair $f,\ g$ of $\fr$ shares $\psi$ but $\fr$ is not normal in $D$.
\end{example}
We thank the referee for suggesting that by using a result of Xu (cf. Theorem D
below) and Corollary \ref{Cor} one can relax condition ($1$) in Theorem \ref{Main1 function thesis} to multiplicity at least $3$. We could improve this further. We state the improved result as follows.
\begin{theorem}\label{M T}
  Let $\fr$ be a family of meromorphic functions defined on a domain $D\subset \C$ and let $\psi\not\equiv0,\infty$ be  a meromorphic function in $D$. Suppose that
  \begin{enumerate}
    \item {every zero of $f$ has multiplicity at least $2$,}
    \item {every pole of $f$ has multiplicity at least $3$,}
    \item { at the common zeros or poles of $f$ and $\psi$, the multiplicity of $f$ is larger than that of $\psi$,}
    \item {$f$ and $g$ share $\psi$, for each pair $f$ and $g$ in $\fr$,}
  \end{enumerate}
  then $\fr$ is normal in $D$.
\end{theorem}
The following example shows that the condition on common zeros is necessary in Theorem \ref{M T}.
\begin{example}
  Let $k, m$ be  two integers such that $2\leq k\leq m$, let $D:=\{z:|z|<1\}, \psi(z)=z^m$ and
  \begin{equation*}
    \fr=\left\{f_n(z)=(n+2)z^k: z\in D, n\in\N\right\}.
  \end{equation*}For each $f_n\in\fr,$ we have \begin{enumerate}
                                                 \item $f_n$ has zeros of multiplicity $k\geq2$,
                                                 \item $f$ has no pole,
                                                 \item for each $i, j$, $f_i$ and $f_j$ share $\psi$ in $D$,
                                                 \item  at the common zero of $f_n$ and $\psi$ the multiplicity of $\psi$ is larger  than or equal to the multiplicity of $f_n$.
                                               \end{enumerate} But $\fr$ is not normal in $D$.
\end{example}
The following example shows that the condition on common poles is necessary in Theorem \ref{M T}.
\begin{example}\label{example Xu}
  Let $k, m$ be  two  integers such that $3\leq k\leq m$, let $D:=\{z:|z|<1\}, \displaystyle{\psi(z)=\frac{1}{z^{m}}}$ and
  \begin{equation*}
    \fr=\left\{f_n(z)=\frac{1}{(n+2)z^k}: z\in D, n\in\N\right\}.
  \end{equation*}For each $f_n\in\fr,$ we have \begin{enumerate}
                                                 \item $f_n$ has poles of multiplicity $k\geq 3$,
                                                 \item $f$ has no zero,
                                                 \item for each $i, j$, $f_i$ and $f_j$ share $\psi$ in $D$,
                                                 \item at the common pole   of $f_n$ and $\psi$ the multiplicity of $\psi$ is larger than or equal to that of $f_n$.
                                               \end{enumerate} But $\fr$ is not normal in $D$.
\end{example}

The following result was proved by Xu \cite{Xu 05}.
\begin{thme}\cite{Xu 05}
   Let $\fr$ be a family of meromorphic functions defined in a domain $D\subset \C$ and let $\psi(\not\equiv0)$ be  a meromorphic function in $D$. For every $f\in\fr$, if
  \begin{enumerate}
    \item[(i)] {$f$ has only multiple zeros,}
    \item[(ii)] {the poles of $f$ have multiplicity at least $3$,}
    \item [(iii)]{ at the common poles of $f$ and $\psi$, the multiplicity of $f$ does not equal the multiplicity of $\psi$,}
    \item[(iv)] {$f(z)\neq\psi(z)$, }
  \end{enumerate}
  then $\fr$ is normal in $D$.
 \end{thme}

We find that Theorem D is not true. Example \ref{example Xu} shows that the analysis was not completed  in Theorem D. Theorem D can be stated as follows in its correct formulation. It can be seen easily that now the proof of Xu \cite{Xu 05} works smoothly.
\begin{thme'}\label{Xu 05 1}
  Let $\fr$ be a family of meromorphic functions defined in a domain $D$. Let $\psi$ $($$\not\equiv0,\infty$$)$ be a function meromorphic in $D$. For every function $f\in \fr$, if
  \begin{enumerate}
    \item {every zero of $f$ has multiplicity at least $2$,}
    \item {every pole of $f$ has multiplicity at least $3$,}
    \item { at the common  poles of $f$ and $\psi$, the multiplicity of $f$ is larger than that of $\psi$,}
    \item {$f(z)\neq \psi(z)$,}
  \end{enumerate}
  then $\fr$ is normal in $D$.
\end{thme'}
In \cite{Xu 05} the meaning of $f(z) \neq  \psi(z)$ is not defined. But from condition (iii) in Theorem D it is clear that it is meant to mean that the meromorphic function $f-\psi$  has no zeros. As  $\psi$ and $f$ might have poles, this is a weaker condition than $f(z_0)\neq \psi(z_0)$ for all $z_0 \in \C$. In particular, Theorem D$'$ is not an immediate corollary to Theorem 1.6 (apart from the fact that we are using Theorem D$'$ to prove Theorem 1.6).
\section{Proof of Main Theorems }
We need some preparation for proving our main result. Zalcman proved a striking result that studies consequence of  non-normality ~\cite{Zalc}. Roughly speaking, it says that in an infinitesimal scaling the family gives a non-constant entire function under the compact-open topology. We  state this renormalization result which has now come to be known as {\it Zalcman's Lemma}.
\begin{zl}~\cite{Zalc}{\ A family $\mathcal F$ of functions meromorphic $($holomorphic$)$ on the unit disc $\Delta$ is not normal if and only if there exist
\begin{enumerate}
\item[$(a)$]{a number r, $0<r<1$}
\item[$(b)$]{points $z_j, |z_j|<r$}
\item[$(c)$]{functions $\{f_j\}\subseteq \mathcal F$}
\item[$(d)$]{numbers $\rho_j\rightarrow0^+$}
 \end{enumerate}
such that
\begin{equation}\notag
f_j(z_j+\rho_j\zeta)\rightarrow g(\zeta)
\end{equation}
spherically uniformly $($uniformly$)$ on compact subsets of $\C$, where $g$ is a non-constant meromorphic $($entire$)$ function on $\C$.  }\end{zl}

Before proving Theorem \ref{Main1 function thesis} we prove some auxiliary results.
\begin{lemma}\label{Main1 function thesis Lemma1}
Let $f$ be a transcendental meromorphic function  and let  $p\not \equiv0$ be a polynomial. Suppose that every zero of $f$ has multiplicity at least $2\ (\text{or} \ 3)$ and every pole of $f$ has multiplicity at least $3\ (\text{resp.} \ 2)$.  Then $f-p$ has infinitely many zeros.
\end{lemma}
\begin{proof}
 Clearly $p(z)$ satisfies $T(r, p(z))=o \{T(r, f(z))\}$. Suppose that $f(z)-p(z)$ has only finitely many zeros. Then  by invoking the second fundamental theorem of Nevanlinna for three small functions $a_1(z)=0, a_2(z)=\infty$ and $a_3(z)=p(z)$, we get
\begin{align*}
  (1+o(1))T(r, f)&\leq \overline{N}(r, f) +\overline{N}\left(r, \frac{1}{f}\right)+\overline{N}\left(r, \frac{1}{f-p}\right)+ S(r, f) \\
   & = \overline{N}(r, f) +\overline{N}\left(r, \frac{1}{f}\right) + S(r, f)\\
   & \leq \frac{N(r, f)}{3} + \frac{N\left(r, \frac{1}{f}\right)}{2} +S(r, f)\\
   & \leq \frac{5}{6}T(r, f) +S(r, f),
\end{align*}
  which is a contradiction.
\end{proof}
\begin{lemma}\label{Main1 function thesis Lemma2}
Let $f$ be a non-constant rational function and let $p\not \equiv0$ be a polynomial of degree at most $m$, where $m$ is a fixed positive integer. Suppose that every zero of $f$ has multiplicity at least $2(m+1)\ (\text{or} \ 3(m+1))$ and every pole of $f$ has multiplicity at least $3(m+1)\ (\text{resp.} \ 2(m+1))$.  Then $f-p$ has at least two distinct zeros.
\end{lemma}
\begin{proof}
For the sake of convenience, we fix the degree of polynomial $p$ as $m$ (deg $p=m$). For deg $p< m$, this proof works verbatim. Now we discuss the following cases.\\

\underline{Case 1.} Suppose $f$ is a non-constant polynomial, then we write
 \begin{equation}\label{Ann eq1p}
    f(z)= A(z-\al_1)^{m_1}\ldots(z-\al_s)^{m_s},
  \end{equation}
  where $A$ is a non-zero constant, $m_i\geq 2(m+1)$ are integers. Then by the fundamental theorem of algebra $f-p$ has zeros. Assume $f-p$ has exactly one zero at $z_0$ and thus we can write
  \begin{equation}\label{Ann eq2p}
    f(z)- p(z)=B(z-z_0)^l,\ l\geq 2(m+1).
  \end{equation}
   Now differentiating \eqref{Ann eq1p} and \eqref{Ann eq2p} $m$ times, we get
   \begin{equation}\label{Ann eq 3p}
     f^{(m)}(z)= (z-\al_1)^{m_1-m}\ldots(z-\al_s)^{m_s-m}g(z),
   \end{equation}
   where $g(z) $ is a polynomial with deg $g(z)\leq m(s-1).$ And
   \begin{equation}\label{Ann eq 4p}
     f^{(m)}(z)- C= B_1(z-z_0)^{l-m},
   \end{equation}
   where $C$ and $B_1$ are non-zero constants.  It is easy to see that $z_0\neq \al_i$ for any $i\in\{1, \ldots s\}$, otherwise $C=0$. Again differentiating \eqref{Ann eq 3p} and \eqref{Ann eq 4p}, we get
   \begin{equation}\label{Ann eq 3ap}
     f^{(m+1)}(z)= (z-\al_1)^{m_1-m-1}\ldots(z-\al_s)^{m_s-m-1}g_1(z),
   \end{equation}
    where $g_1(z) $ is a polynomial with deg $g_1(z)\leq (m+1)(s-1).$ And
   \begin{equation}\label{Ann eq 4ap}
     f^{(m+1)}(z)= B_2(z-z_0)^{l-m-1},
   \end{equation}
   where $B_2$ is a non-zero constant. From \eqref{Ann eq 4ap}, we see that $f^{(m+1)}(\al_i)\neq0,$ for $i=1,\ldots, s$. This shows that the multiplicity of zeros of $f$ is at most $m$, which is a contradiction. \\

   \underline{Case 2.} Suppose $f$ is a non-polynomial rational function, then we set
   \begin{equation}\label{Ann eq 1r}
     f(z)=A\frac{(z-\al_1)^{m_1}\ldots (z-\al_s)^{m_s}}{(z-\beta_1)^{n_1}\ldots (z-\beta_t)^{n_t}},
   \end{equation}
   where $A$ is a non-zero constant, $m_i\geq 2(m+1)\quad (i=1, 2, \ldots, s)$ and $n_j\geq 3(m+1)\quad (j=1, 2, \ldots, t).$\\
   Let us define
   \begin{equation}\label{Ann A}
     \sum_{i=1}^{s}m_i=M\geq2(m+1)s \ \text{and}\ \sum_{j=1}^{t}n_j=N\geq 3(m+1)t.
        \end{equation}
On differentiating   \eqref{Ann eq 1r} $m+1$ times, we get
   \begin{equation}\label{Ann eq 2r}
     f^{(m+1)}(z)=A_1\frac{(z-\al_1)^{m_1-m-1}\ldots (z-\al_s)^{m_s-m-1}h(z)}{(z-\beta_1)^{n_1+m+1}\ldots (z-\beta_t)^{n_t+m+1}},
   \end{equation}where $A_1$ is a non-zero constant and $h(z)$ is a polynomial with deg $h(z)\leq (m+1)(s+t-1)$. \\Now we discuss the following cases.\\

   \underline{Case 2a.} If $f-p$ has no zeros, then we can write
   \begin{equation}\label{Ann eq 1ar}
     f(z)=p(z)+\frac{A_2}{(z-\beta_1)^{n_1}\ldots (z-\beta_t)^{n_t}},
   \end{equation}
   where $A_2$   is a non-zero constant. We notice that \eqref{Ann eq 1r} and \eqref{Ann eq 1ar} together gives $M\geq N$.
   On differentiating  \eqref{Ann eq 1ar} $m+1$ times, we get
   \begin{equation}\label{Ann eq 2ar}
     f^{(m+1)}(z)=\frac{h_1(z)}{(z-\beta_1)^{n_1+m+1}\ldots (z-\beta_t)^{n_t+m+1}},
   \end{equation}where $h_1(z)$ is a polynomial with deg $h_1(z)\leq(m+1)t.$\\
Also, by \eqref{Ann eq 2r} and \eqref{Ann eq 2ar}, we have $M-(m+1)s\leq (m+1)t,$ which gives that $M\leq (m+1)(s+t)$ and  combining this with \eqref{Ann A} we get
    \begin{equation}\notag
      M\leq(m+1)(s+t)\leq \frac{5}{6}M< M,
    \end{equation}which is a contradiction.\\
Here we note that when the multiplicity of poles of $f$ is $\geq 2(m+1)$ and the multiplicity of zeros of $f$ is $\geq 3(m+1)$, the proof is exactly the same.\\

\underline{Case 2b.} If $f-p$ has exactly one zero at $z_0$, then we can write
 \begin{equation}\label{Ann eq 1arr}
     f(z)=p(z)+\frac{C_1(z-z_0)^l}{(z-\beta_1)^{n_1}\ldots (z-\beta_t)^{n_t}},
   \end{equation}
   where $C_1$   is a non-zero constant and $l$ is a positive integer. On differentiating  \eqref{Ann eq 1arr} $m+1$ times, we get
\begin{equation}\label{Ann eq 2arr}
     f^{(m+1)}(z)=\frac{(z-z_0)^{l-m-1}h_2(z)}{(z-\beta_1)^{n_1+m+1}\ldots (z-\beta_t)^{n_t+m+1}},
   \end{equation}where $h_2(z)$ is a polynomial with deg $h_2(z)\leq(m+1)t.$ Note that  \eqref{Ann eq 2arr} also holds in the case $l\leq m$. In this case\begin{equation*}
     f^{(m+1)}(z)=\frac{h_2(z)}{(z-\beta_1)^{n_1+m+1}\ldots (z-\beta_t)^{n_t+m+1}}.
   \end{equation*}\\

   Now, we claim that  $z_0\neq\al_i$ for any $i\in\{1, \ldots, s\}$. Suppose that $z_0=\al_i$ for some $i\in\{1,\ldots,s\}$. If $l\geq m+1$, this would mean that $z_0$ is a zero of order at least $m+1$ of $p$. And if $l\leq m$, then from \eqref{Ann eq 1arr}, $z_0$ is a zero of $p$ with multiplicity at least $ l$. Now from \eqref{Ann eq 1arr} we have
\begin{equation}
     f_1(z)=p_1(z)+\frac{C_2}{(z-\beta_1)^{n_1}\ldots (z-\beta_t)^{n_t}},
   \end{equation}
    where $C_2$ is a constant, $f_1=f/(z-z_0)^l$ and $p_1=p/(z-z_0)^l$. Now proceed as in the Case 2a and get a contradiction. Hence, we have $z_0\neq\al_i$ for any $i\in\{1, \ldots, s\}$.\\

Now, we discuss the following two subcases.\\

\underline{Subcase 2b.1.} If $M\geq N.$ Since $z_0\neq\al_i \ \text{for any}\ i\in\{1,\ldots, s\},$ therefore by \eqref{Ann eq 2r} and \eqref{Ann eq 2arr}, we have $M-(m+1)s\leq (m+1)t.$ Which further implies that $M\leq (m+1)(s+t)$ and  combining this with \eqref{Ann A} we get
    \begin{equation}\notag
      M\leq(m+1)(s+t)\leq \frac{5}{6}M< M,
    \end{equation}which is a contradiction.\\

    \underline{Subcase 2b.2.}  If $M<N$, then by \eqref{Ann eq 2r} and \eqref{Ann eq 2arr}, we deduce that $l-m-1\leq (m+1)(s+t-1)$, which gives that
    \begin{equation}\notag
     l\leq(m+1)(s+t)\leq \frac{5}{6}N<N.
    \end{equation} But by \eqref{Ann eq 1r} and \eqref{Ann eq 1arr}, we have $M\leq \max\{N+\text{deg}\ p, l\},$ with equality if $N+ \text{deg}\ p\neq l$. So $l<N$ leads to the contradiction $M>N$.
      \end{proof}
    Now we are ready to prove Theorem \ref{Main1 function thesis}.
    \begin{proof}[Proof of Theorem \ref{Main1 function thesis}]
     Without loss of generality, we may assume that $D=\{z\in\C: |z|<1\}$. Suppose on the contrary that $\fr$ is not normal at $z_0=0$. Now we have two cases to consider.

     \underline{Case 1.} Suppose $\psi(0)\neq0.$ Then by Zalcman's Lemma there exist
     \begin{enumerate}
\item{ a sequence of complex numbers $z_j \rightarrow z_0=0$, $|z_j|<r<1$},
\item{ a sequence of functions $f_j\in \mathcal F$, }
\item{ a sequence of positive numbers $\rho_j \rightarrow 0$},
\end{enumerate} such that $g_j(\xi)=f_j(z_j+\rho_j\xi)$ converges locally uniformly with respect to the spherical metric to a non-constant meromorphic function $g(\xi)$. It is evident from Hurwitz's theorem that $g$ satisfies the following properties:
\begin{enumerate}
  \item[$($a$)$] every zero of $g$ has multiplicity at least $2(m+1)\ (\text{or}\ 3(m+1)),$
  \item[$($b$)$] every pole of $g$ has multiplicity at least $3(m+1)\ (\text{resp.}\ 2(m+1)).$
\end{enumerate}
  Also on every compact subsets of $\C$, not containing poles of $g$,  we get that
\begin{align}
f_j(z_j+\rho_j\xi)-\psi(z_j+\rho_j\xi)&= g_j(\xi)-\psi(z_j+\rho_j\xi)\notag\\
&\rightarrow g(\xi)-\psi(0).\label{Ann eq proof 1}
\end{align}
Clearly, $g(\xi)-\psi(0)\not\equiv0$. Therefore by Lemma \ref{Main1 function thesis Lemma1} and Lemma \ref{Main1 function thesis Lemma2}, we know that $g(\xi)-\psi(0)$  has at least two distinct zeros. Let $w_1$ and $w_2$ be two distinct zeros of $g(\xi)-\psi(0)$. We can find two disjoint neighborhoods $N_{\delta_1}(w_1)=\{z: |z-w_1|<\delta_1\}$ and $N_{\delta_2}(w_2)=\{z: |z-w_2|<\delta_2\}$ such that $N_{\delta_1}(w_1)\cup N_{\delta_2}(w_2)$  contains no  zero of $g(\xi)-\psi(0)$ other than $w_1$ and $w_2$. By Hurwitz's theorem, there exist two sequences  $\{w_{1_j}\}\subset N_{\delta_1}(w_1), \{w_{2_j}\}\subset N_{\delta_2}(w_2)$ converging to $w_{1},   w_{2}$ respectively and  for sufficiently large $j$, we have
\begin{align*}
  f_j(z_j+\rho_jw_{1_j})-\psi(z_j+\rho_jw_{1_j}) &=0,\\
  f_j(z_j+\rho_jw_{2_j})-\psi(z_j+\rho_jw_{2_j}) &=0.
  \end{align*}
  Since each pair $f_a, f_b$ of $\fr$ shares $\psi$ IM in $D$, therefore for any positive integer $m$ we have
  \begin{align*}
  f_m(z_j+\rho_jw_{1_j})-\psi(z_j+\rho_jw_{1_j}) &=0,\\
  f_m(z_j+\rho_jw_{2_j})-\psi(z_j+\rho_jw_{2_j}) &=0.
  \end{align*}
  Fixing $m$ and taking $j\rightarrow \infty$, we see that $z_j+\rho_jw_{1_j}\rightarrow 0$, $z_j+\rho_jw_{2_j}\rightarrow 0$ and $f_m(0)-\psi(0)=0$. Since the zero set is discrete, for large values of $j$ we have
  \begin{equation}\notag
    z_j+\rho_jw_{1_j}=0=z_j+\rho_jw_{2_j},
  \end{equation} hence
  \begin{equation}\notag
    w_{1_j}=-\frac{z_j}{\rho_j}=w_{2_j}.
  \end{equation}
  This contradicts the fact that $N_{\delta_1}(w_1)\cap N_{\delta_2}(w_2)=\emptyset.$  \\

  \underline{Case 2.} Let $\psi(0)=0.$ We can write $\psi(z)=z^t \phi(z),$ where $t (\leq m)$ is a positive integer and $\phi(z)$ is a holomorphic function in $D$ such that $\phi(0)\neq0$. Now we consider the following subcases:\\

\underline{Subcase 2.1.} If $f(0)\neq \psi(0)$, for some $f\in \fr$. Then, there exists $r>0$ such that $f(z)\neq 0$ and $f(z)\neq \psi(z)$, for all $z\in N_r(0)$ and for all $f\in\fr$. Then, normality is confirmed by Theorem D$'$.\\

\underline{Subcase 2.2.} If $f(0)=\psi(0)$, for some $f\in \fr$. Then, there exists $r>0$ such that
$f(z)\neq \psi(z)$, for all $z\in N_r'(0)=\{z: 0<|z|<r\}$ and for all $f\in\fr$. Now consider the family $\mathcal{G}=\left\{f(z)/z^t: f\in\fr\right\}$ and the function $\phi(z)$. On $N_r(0)$, $\mathcal{G}$ satisfies Subcase 2.1, therefore $\mathcal{G}$ is normal in $N_r(0)$.
 Now, we show that $\fr$ is normal in $N_r(0)$. Clearly, $\fr$ is normal in $N_r'(0)$ and  $g(0)=0$, for all $g\in\mathcal{G}$. So there exists $0<\delta $($<r$) such that if $g\in \mathcal G$, $|g(z)|\leq 1$, for $z\in N_{\delta}(0).$ On the boundary of $N_{\delta}(0)$,  $|f(z)|=|z|^t|g(z)|\leq \delta^t$. Thus, by the maximum principle $|f(z)|\leq \delta^t$ on $N_{\delta}(0),$ for all $f\in\fr.$ So $\fr$ is normal on $N_{\delta}(0)$ and hence on $N_r(0).$
 \end{proof}

Now we give the  proof of Theorem \ref{M T}.
\begin{proof}[Proof of Theorem \ref{M T}]
  Since normality is a local property, it is sufficient to show that $\fr$ is normal at each point  of $D$. Now we consider the following cases to check the normality at an arbitrarily chosen  point $z_0\in D$.\\

  {{\underline{Case 1.}}} If $f(z_0)\neq \psi(z_0), $ for some $f\in\fr$. Then, there exists $r>0$ such that $f(z)\neq \psi(z)$, for all $z\in N_r(z_0)$ and for all $f\in\fr$. Then $\fr$ is normal at $z_0$, by Theorem D$'$.\\

  {{\underline{Case 2a.}}} If $f(z_0)=\psi(z_0)\neq 0, \infty,$ for some $f\in\fr$. Then, there exists $r>0$ such that $f(z)\neq \psi(z)$ and $f(z)/\psi(z)\neq 0, \infty$, in $N_r'(z_0).$ Now consider the family $\fr_1=\{f(z)/\psi(z): f\in\fr\}$. By Corollary \ref{Cor}, $\fr_1$ is normal at $z_0$. Since each $f/\psi \in \fr_1$ is holomorphic and $\fr_1$ is normal at $z_0$, we get $\fr$ is normal at $z_0$.\\

  {{\underline{Case 2b.}}} If $f(z_0)=\psi(z_0)=0$, for some $f\in\fr$. Then, there exists $r>0$ such that $\psi(z)\neq 0, \infty$ and $f(z)\neq \psi(z)$ in $N_r'(z_0)$.  Let $m$ be the multiplicity of the zero of $\psi$ at $z=z_0$. Consider the family $\mathcal{G}=\left\{f/(z-z_0)^m: f\in\fr\right\}$ and the function $\psi(z)/(z-z_0)^m$. On $N_r(z_0)$, $\mathcal{G}$ satisfies case 1. Therefore $\mathcal{G}$ is normal in $N_r(z_0)$.
 Now, we show that $\fr$ is normal in $N_r(z_0)$. Clearly, $\fr$ is normal in $N_r'(z_0).$ And by condition ($3$) of the theorem, $g(z_0)=0$, for all $g\in\mathcal{G}$. So there exists $0<\delta $($<r$) such that if $g\in \mathcal G$, $|g(z)|\leq 1$, for $z\in N_{\delta}(z_0).$ On the boundary of $N_{\delta}(z_0)$,  $|f(z)|=|z-z_0|^m|g(z)|\leq \delta^m$. Thus, by the maximum principle $|f(z)|\leq \delta^m$ on $N_{\delta}(z_0),$ for all $f\in\fr.$ So $\fr$ is normal on $N_{\delta}(z_0)$ and hence on $N_r(z_0).$\\

  {{\underline{Case 2c.}}} If $f(z_0)=\psi(z_0)=\infty$, for some $f\in\fr$.  Then, there exists $r>0$ such that $f(z)\neq \psi(z)$ in $N_r' (z_0).$ Let $k$ be the multiplicity of the pole of $\psi(z)$ at $z=z_0$. Consider the family $\mathcal H = \{(z-z_0)^kf: f\in\fr\}$ and the function $(z-z_0)^k\psi(z)$. Then $\mathcal{H}$ satisfies case 1, so $\mathcal H$ is normal in $N_r(z_0)$.
 Now, we prove that $\fr$ is normal at $N_r(z_0)$. Clearly, $\fr$ is normal in $N_r'(z_0).$  Also, by condition ($3$) of the theorem, $h(z_0)=\infty,$ for all $h\in \mathcal{H}.$ So there exists $0<\delta$($<r$) such that  $|h(z)|\geq 1$, for all $h\in \mathcal{H}$ and $z\in N_{\delta}(z_0).$ It follows that, $f(z)\neq 0$ in $N_{\delta}(z_0)$, for all $f\in\fr$.  Since $\fr$ is normal in $N_r' (z_0)$, then the family $1/\fr=\left\{1/f: f\in\fr \right\}$ is holomorphic in $N_{\delta}(z_0)$ and normal  in $N'_{\delta}(z_0),$ but it is not normal at $z=z_0$. Thus, there exists  a sequence $\{1/f_n\}\subset 1/\fr$ which converges locally uniformly in $N'_{\delta}(z_0),$ but no subsequence of $\{1/f_n\} $ converges uniformly in a  neighborhood of $z_0$. The maximum modulus principle implies that $1/f_n\rightarrow \infty$ on compact subsets in $N'_{\delta}(z_0).$ Hence, $f_n\rightarrow 0$ uniformly on compact subsets of $N'_{\delta}(z_0)$ and this shows that $h_n\rightarrow 0$ uniformly on compact subsets of $N'_{\delta}(z_0).$ Which is a contradiction to the fact that $|h_n(z)|\geq 1$ in $N_{\delta}(z_0)$.
\end{proof}

{\bf{Acknowledgement.}} We would like to thank the referee for careful reading of our paper and making suggestions for changes which has enhanced the presentability of the paper. Furthermore, we would also thank him/her for bringing to  our attention  the work done by  Xu \cite{Xu 05}.
    

\begin{thebibliography}{00}
    \bibitem{Ahl} L. V. Ahlfors, \emph{Complex Analysis,} Third edition, McGraw-Hill, 1979.
\bibitem{GDSK} G. Datt and S. Kumar, {Normality and sharing functions,} \emph{Indian J. Pure Appl. Math.} 46 (6), (2015),
853--864.
\bibitem{GDSK1} G. Datt and S. Kumar, {Some normality criteria,} \emph{Turkish J. Math.} (accepted for publication), arXiv: 1403.1447[math.CV].
        \bibitem{Hay} W. K. Hayman, {Meromorphic Functions,} \emph{Clarendon Press, Oxford}, 1964.

    \bibitem{Schiff} J. Schiff,  \emph{Normal Families,} Springer-Verlag, Berlin, 1993.

    \bibitem{Sch 92}W. Schwick, Sharing values and normality,\emph{ Arch. Math.} {59},(1992), 50--54, doi: 10.1007/BF01199014.

     \bibitem{Sun 94} D. C. Sun, The shared value criterion for normality, \emph{J. Wuhan Univ. Natur. Sci. Ed.} 3 (1994), 9--12.
\bibitem{Xu 05} Y. Xu, On Montel's theorem and Yang's problem, \emph{J. Math. Anal. Appl.} 305 (2005), 743--751.
\bibitem{Xu 10} Y. Xu, Montel's criterion and shared function, \emph{Publ. Math. Debrecen}, 77/3-4 (2010), 471--478.

\bibitem{Yang}L. Yang, \emph{Value Distribution Theory,} {Springer-Verlag, Berlin,} 1993.
\bibitem{Zalc}L. Zalcman, {Normal families: new perspectives,} \emph{Bull. Amer. Math. Soc.}, {35}, no. 3  (1998), 215--230.
           \end{thebibliography}
    \end{document}